\documentclass[11pt]{amsart}
\usepackage{amscd, amsmath, amsthm, amssymb, xy, xypic}
\newtheorem{theorem}{Theorem}[section]

\newtheorem{lem}[theorem]{Lemma}

\newtheorem{condition}[theorem]{Condition}
\newtheorem{cor}[theorem]{Corollary}
\theoremstyle{defi} 
\newtheorem{defi}[theorem]{Definition}

\theoremstyle{remark}
\newtheorem{remark}[theorem]{Remark}
\numberwithin{equation}{section}

\def \hd #1 {\bfseries #1 \mdseries}
\def \italic #1 {\bfseries \it #1 \rm \mdseries}
\def \ra {\rightarrow}
\def \cen #1 { \begin{center} #1 \end{center}}

\def \mbz {\mathbb Z}

\def \mbc {\mathbb C}
\def \mbp {\mathbb P}

\def \mco {\mathcal {O}}

\def \Q {${\mathbb {Q}}\,$}

\def \Pic {{\rm{Pic}}}
\def \rk {{\rm{rk}}}

\def \Sing {{\rm{Sing}}}

\def \Cl {{\rm{Cl}}}

\def \mult {{\rm{mult}}}
\begin{document}
\title[Lefschetz Hyperplane theorem on \Q-Fano $3$-folds]
{A type of the Lefschetz hyperplane section theorem on \Q-Fano $3$-folds with Picard number one and $\frac{1}{2}(1,1,1)$-singularities}
\author[N.-H. Lee]{Nam-Hoon Lee }
\address{Department of Mathematics Education, Hongik University, 42-1, Sangsu-Dong,
Mapo-Gu, Seoul 121-791, South Korea}
\address{School of Mathematics, Korea Institute for Advanced Study, Dongdaemun-gu, Seoul 130-722, Korea }
\email{nhlee@kias.re.kr}
\dedicatory{Dedicated to  Prof.\ Daniel Burns on the occasion of his 65th birthday}
\begin{abstract}
We prove a type of the Lefschetz hyperplane section theorem on \Q-Fano $3$-folds with Picard number one and $\frac{1}{2}(1,1,1)$-singularities by using some degeneration method. As a byproduct, we obtain a new example of a Calabi--Yau $3$-fold $X$ with Picard number one whose invariants are
$$(H_X^3,\, c_2 (X) \cdot H_X, \,{{e}} (X) ) = (8, 44, -88),$$
where $H_X$, $e(X)$ and $c_2(X)$ are an ample generator of $\Pic(X)$,
the topological Euler characteristic number and the second Chern class of $X$ respectively.
\end{abstract}
\maketitle
\setcounter{section}{-1}
\section{Introduction}
We work over the complex number field. Let us start by the definition of \Q-Fano variety.
\begin{defi}
A projective variety $Y$ is called a \emph{\Q-Fano variety} if $Y$ has only terminal singularities and $-K_Y$ is ample.
\end{defi}
As  the smooth Fano $3$-folds, \Q-Fano varieties with Picard number one form an important class among general ones.
In this note, we investigate a type of the Lefschetz hyperplane section theorem on certain \Q-Fano $3$-folds.
Let $Y$ be a \Q-Fano $3$-fold with Picard number one and only cyclic singularities of type $\frac{1}{2}(1,1,1)$.
Such \Q-Fano $3$-folds have been extensively investigated in \cite{Sa1, Sa2},   \cite{CaFl}  and  \cite{Ta1, Ta2}.
We assume the following:

\begin{condition} \label{newcon}\
\begin{enumerate}
\item The divisor class group $\Cl(Y)$ of $Y$ is generated a single element $h$.
\item $-2K_Y$ is very ample.
\item The linear system $|-K_Y|$ has a member $D$ that is smooth outside $\Sing(Y)$ and has singularities of type $\frac{1}{2}(1,1)$ at $\Sing(Y)$.
\end{enumerate}
\end{condition}
Note that all the examples of \Q-Fano $3$-folds classified by H. Takagi satisfy this condition (Corollary 3.4 in \cite{Ta2}).

 Let $Y_n$ be a Weil divisor of $Y$ such that $Y_n \sim n h$. If $n$ is odd, then $Y_n$ can not be smooth and at best it has singularities of type $\frac{1}{2}(1,1)$ at $\Sing(Y)$ -- call such one quasi-smooth.
If $n$ is even, $Y_n$ can be smooth.
The purpose of this note is to prove the following type of the Lefschetz hyperplane section theorem on $Y$.

\begin{theorem} \label{mainthm}
Let $Y$ be a \Q-Fano $3$-fold with Picard number one and only $\frac{1}{2}(1,1,1)$-singularities, satisfying Condition \ref{newcon}, and $Y_n$ be the smooth or quasi-smooth divisor  on $Y$ defined as above.
Then
the homomorphism of divisor class groups
$$\Cl(Y) \ra \Cl(Y_n)_f$$
 is injective and primitive (i.e.\ the cokernel is torsion-free).
For an abelian group $A$, we set $A_f := A/ (\rm{torsion\,\,\, part\,\,\, of }$ $A)$.
\end{theorem}

As a byproduct of the proof, we will be able to calculate some invariants of a Calabi--Yau $3$-fold, which was constructed in \cite{Le1} but whose invariants could not be obtained there.
It turns out that it is a new example of Calabi--Yau $3$-fold with Picard number one (\S \ref{sec4}).
Its invariants are
$$(H_X^3, c_2 (X) \cdot H_X, {{e}} (X) ) = (8, 44, -88),$$
where $H_X$, $e(X)$ and $c_2(X)$ are an ample generator of $\Pic(X)$,
the topological Euler characteristic number and the second Chern class of $X$ respectively.

The idea of the proof is to use a \emph{smooth} double covering $X$ over $Y$.
Firstly we determine the Picard group of $X$ by some degeneration method. Then we can calculate invariants of the preimage of $Y_n$ by the ordinary Lefschetz hyperplane section theorem on $X$. With this we finally obtain the invariants of $Y_n$.

In \S \ref{sec2}, we briefly introduce some degeneration method from \cite{Le2}. \S \ref{sec3} is devoted to calculation of the Picard group of the blow-up of $Y$ at its singularities.
In \S \ref{sec4}, we apply the degeneration method  to calculate invariants of $X$  and prove the main theorem.

\section{Degeneration and integral cohomology groups}\label{sec2}\

We summarize some results of \cite{Le2} that will be used in this note.

Let $\phi: W \ra \bar \Delta$ be a proper map from a smooth $(n+1)$-fold $W$ with
boundary onto a closed disk $\bar \Delta = \{ t \in \mbc \mid \|t\|
\leq 1\}$ such that the fibers $W_t = \phi^{-1} (t)$ are connected
K\"ahler $n$-folds for every $t \neq 0$ (generic) and the central
fiber $W_0=\phi^{-1} (0) =\bigcup_\alpha V_\alpha$ is a normal
crossing of $n$-folds. We denote the general fiber by $W_t$. The
condition, $t \neq 0$, is assumed in this notation. We call such a
map $\phi$ (or simply the total space $W$) a semi-stable degeneration
of $W_t$ and we say that $W_0$ is smoothable to $W_t$ with the
smooth total space.
Let $V_{ij} = V_i \cap V_j$, $V_{ijk} = V_i \cap V_j \cap V_k$ and etc. Consider
the exact sequence
$$0 \ra \mbz_{W_0} \ra \bigoplus_\alpha \mbz_{V_\alpha} \ra \bigoplus_{i < j}
\mbz_{V_{ij}} \stackrel{\tau}\ra \bigoplus_{i < j < k} \mbz_{V_{ijk}} {\ra}
\cdots $$
We introduce a short exact sequence,
$$0 \ra \mbz_{W_0} \ra \bigoplus_\alpha \mbz_{V_\alpha} \ra {{\rm{ker}}} (\tau)
{\ra} 0 $$
to give an exact sequence
\begin{align}
\cdots \ra H^m (W_0, \mbz) \stackrel{\psi_m}{\ra}
\bigoplus_\alpha H^m (V_\alpha, \mbz) \ra
H^{m} ({{\rm{ker}}} (\tau)) \ra \cdots \label{seq}
\end{align}
Note that $\psi_m = \bigoplus_\alpha
j_\alpha^*|_{H^m (W_0, \mbz)}$, where $j_\alpha: V_\alpha
\hookrightarrow W_0$ is the inclusion.
Let
$$G^m (W_0, \mbz) = {\rm{im}} (\psi_m) \subset \bigoplus_\alpha H^m (V_\alpha, \mbz).$$

Consider non-negative integers, $q_1, \cdots, q_k$ such that
$$q_1 + \cdots +q_k = n.$$
Then there are multilinear maps, defined by the cup-product:
$$H^{2q_1} (W_t, \mbz) \times \cdots \times H^{2q_k} (W_t, \mbz) \ra \mbz,$$
$$H^{2q_1} (V_\alpha, \mbz) \times \cdots \times H^{2q_k} (V_\alpha, \mbz) \ra \mbz.$$
The latter one induces a multilinear map
$$\bigoplus_{\alpha_1} H^{2q_1} (V_{\alpha_1}, \mbz) \times \cdots \times \bigoplus_{\alpha_k} H^{2q_k} (V_{\alpha_k}, \mbz) \ra \mbz,$$
setting the mixed terms to be equal to zero. By restricting, we can define a multilinear map
$$G^{2q_1} (W_0, \mbz) \times \cdots \times G^{2q_k} (W_0, \mbz) \ra \mbz.$$
Note that this multilinear map define a cup product on $G^{2} (W_0, \mbz)$.

Let $h^{p,q}(W_t)=\dim H^{p,q}(W_t)$ and $h^{i}(W_t)=\dim H^{i}(W_t)$.
We will need the following simplified results (\cite{Le2}, Theorem 6.5):
\begin{theorem} \label{degthm} Suppose that $h^{2,0}(W_t)=0$.
Let $a_1, a_2, \cdots, a_k \in G^2(W_0, \mbz)$ with $k=h^2(W_t)$ and $b_1, b_2, \cdots, b_k \in G^{n-2}(W_0, \mbz)$ such that the intersection matrix of cup products:
$$(a_i \cdot b_j)$$
is unimodular. Then the sublattice $\langle a_1, \cdots, a_m \rangle$ of $G^{2} (W_0, \mbz)$ is isomorphic to $H^2(W_t, \mbz)_f$ with the cup product preserved.
\end{theorem}

\section{Picard group of some blow-up}\label{sec3}\
From now on, the variety $Y$ is a \Q-Fano $3$-fold as in Theorem \ref{mainthm}.  Under (1) in Condition \ref{newcon}, the class group $\Cl(Y)$ is generated by an element $h$ with  $h^3 > 0$.
Let $f : V_1 \ra Y$ be the blow up at $\Sing(Y)=\{p_1, \cdots, p_N\}$ and $e_i$'s be the exceptional divisors over $p_i$'s.
Let $B$ be a prime divisor on $Y$. Then we have
$$f^*(B) = \widetilde B+\sum_i \frac{1}{2} q_i e_i,$$
for some nonnegative integer $q_i$, where $\widetilde B$ is the proper transform of $B$.
We define the number $q_i$ as the multiplicity $\mult_p(B)$ of $B$ at $p_i \in \Sing(Y)$.
It is easy to see that $B$ is Cartier at $p_i$ if and only if $\mult_{p_i}(B)$ is even.
The purpose of this section is to show:
\begin{lem}\label{key} The Picard group of $V_1$ is
$$\Pic(V_1) = \langle f^*(h)- \frac{1}{2}\sum_i e_i, e_1, \cdots, e_N \rangle.$$
\end{lem}
\begin{proof}
Since $\dim Y$ is odd, $K_Y $ is \emph{not} Cartier at $p_i \in \Sing(Y)$.
Let $K_Y \sim -r h$ for some integer $r$. Note that $r$ is odd since otherwise $K_Y$ would be Cartier.
Note that $h \sim (r+1)h + K_Y$ and $(r+1)h$ is Cartier. So $h$ is not Cartier at any point $p_i$.

Note that $f^*(K_Y)+ \frac{1}{2}\sum_i e_i=K_{V_1}$. So the following rational combination of divisors
\begin{align*}
f^*(h)- \frac{1}{2}\sum_i e_i &\sim f^*((r+1)h)+ f^*(K_Y)+ \frac{1}{2}\sum_i e_i - \sum_i e_i \\
&\sim f^*((r+1)h)+K_{V_1}- \sum_i e_i
\end{align*}
is actually a divisor class since $(r+1)$ is even. Consider a subgroup
$$G:= \langle f^*(h)- \frac{1}{2}\sum_i e_i, \,\, e_1, \cdots, e_N \rangle$$
of $\Pic(Y)$.
Let $L$ be a prime divisor on $V_1$ such that $e_i \not \subset L$ for any $i$. Let $q_i$ be the degree $L|_{e_i}$ in $e_i$.
Then we have
$$L = f^*(\check{L}) -\sum_i \frac{1}{2} q_i e_i,$$
where $\check L = f(L)$. Note that $q_i = \mult_{p_i}(\check L)$.
We can set $\check L \sim k h$ for some integer $k$.
Note that the following are equivalent:
\begin{enumerate}
\item $q_i$ is even,
\item $\check L$ is Cartier at $p_i$,
\item $kh$ is Cartier at $p_i$,
\item $k$ is even.
\end{enumerate}
So we have
\cen{$k \equiv q_i (\rm{mod} 2)$}
for any $i$. Then
\begin{align*}
L &= f^*(\check{L}) -\sum_i \frac{1}{2} q_i e_i \\
&\sim k f^*(h) - \sum_i \frac{1}{2} q_i e_i\\
&= k \left ( f^*(h)- \frac{1}{2}\sum_i e_i \right ) + \sum_i \frac{1}{2} (q_i -k) e_i.
\end{align*}
Since $(q_i -k)$'s are even, $\frac{q_i-k}{2}$'s are integers and we have $L \in G$.
A divisor is linear equivalent to a linear combination of prime divisors, so any divisor class belongs to $G$.
Therefore we have $\Pic(V_1) \subset G$.
\end{proof}

\begin{remark}
If $\Cl(Y)$ has non-zero torsion part or its rank is higher than one, the computation of $\Pic(V_1)$ is non-trivial and depends on the detailed geometry of $Y$.
\end{remark}
Since {$h^i (Y, \mco_Y) = 0$ for $ 0< i < 3$,}
we have
$$H^2(V_1, \mbz) \simeq \Pic(V_1) = \langle f^*(h)- \frac{1}{2}\sum_i e_i, e_1, \cdots, e_N \rangle.$$
By Poincar\'e duality, there are classes $m_0, m_1, \cdots, m_N$ in $H^{4}(V_1, \mbz)$ whose cup product matrix with $f^*(h)- \frac{1}{2}\sum_i e_i, e_1, \cdots, e_N$ is the $(N+1)\times (N+1)$ identity matrix.

\section{The proof} \label{sec4}
Under (2) in Condition \ref{newcon},  one can find a smooth surface $S$ in $|-2K_Y|$ such that $S \cap Y_n$ and $C:=S\cap D$ are are all smooth curves.
Then
there is a Calabi--Yau $3$-fold $X$ of Picard number one (Theorem 2.1 of \cite{Le1}) that is a double covering over $Y$ with
the
branch locus
\cen{$ S \cup {\rm{Sing}} (Y)$}
(\cite{Le1}, Theorem 1.1).

Now we use a semistable degeneration $W \ra \Delta$ of $X$, which we constructed in the proof of Theorem 2.1 of \cite{Le1} (p.539 - 540).
The generic fiber $W_t$ is a deformation of $X$.

To describe the
central fiber $W_0$, we introduce some notation.
\begin{enumerate}
\item[-] Let $V_1$ be the blow-up of $Y$ at $\Sing (Y)$ as before. Take
surfaces $S$, $D$ as above. Note that the proper transform $\widetilde D$
is a smooth $K3$ surface.
\item[-] Let $g : V_2 \ra Y$ be the blow-up at $\Sing (Y)$ and $C := S\cap D$,
and $f_1, \dots, f_N$ the $g$-exceptional divisors over $\Sing (Y)$. Let
$\widetilde {D'}$
be the proper transform of $D$.
\item[-] Let $E_1, \dots, E_N$ be copies of $\mbp^3$'s.

\end{enumerate}

Then $W_0 = V_1 \cup V_2 \cup E_1 \cup \cdots \cup E_N,$ where $V_1 \cap V_2 = \widetilde D = \widetilde {D'}$,
$V_1 \cap E_i = e_i$, $V_2 \cap E_i = f_i$, $e_i \neq f_i$ $(1 \leq i \leq N)$, and $E_i \cap E_j = \emptyset$
$(1 \leq i < j \leq N)$. $e_i$ and $f_i$ are planes of $E_i$.
Note that $e_i$ and $f_i$ were denoted by $H_{1i}$ and $H_{2i}$ respectively in \cite{Le1}.

For simplicity, let ${\rm{Sing}} (Y) = \{ p \}$ be composed of a single point ($N=1$). The proof for general cases is similar.
Let $q=2$ or $4$ and abbreviate $H^*(*, \mbz), G^*(*, \mbz)$ as in \S \ref{sec2} by $H^*(*), G^*(*)$ respectively.
Since
$$H^{q-1}(\widetilde D) = H^{q-1} (e_1 \cap f_1)=0,$$
the Mayer--Vietories sequences are:
$$0 \ra H^q(V_1 \cup V_2) \ra H^q(V_1) \oplus H^q( V_2) \ra H^q(\widetilde D)$$
and
$$0 \ra H^q(e_1 \cup f_1) \ra H^q(e_1) \oplus H^q( f_1) \ra H^q(e_1 \cap f_1).$$

So we can recognize
\cen{ $H^q(V_1 \cup V_2)$ and $H^q(e_1 \cup f_1) $}
as sublattices of
\cen{$H^q(V_1) \oplus H^q( V_2)$ and $H^q(e_1) \oplus H^q( f_1)$}
respectively. Let us denote them by
\cen{$\widehat{H^q(V_1) \oplus H^q( V_2)}$ and $\widehat{H^q(e_1) \oplus H^q( f_1)}$}
respectively.

Consider the following commutative diagram:

$$
\xymatrix{
H^q (W_0) \ar[r]^{\psi_2} \ar@{-->}[d]& H^q(V_1) \oplus H^q(V_2) \oplus H^q(E_1) \\
H^q(V_1 \cup V_2) \oplus H^q(E_1) \ar[ur] \ar@{=}[r] \ar@{-->}[d]& \left(\widehat{H^q(V_1) \oplus H^q ( V_2)}\right ) \oplus H^q(E_1) \ar@{^{(}->}[u]_{i} \ar[d]^\nu\\
H^q(e_1 \cup f_1) \ar@{=}[r]& \widehat{H^q(e_1) \oplus H^q( f_1)}
}$$
where the sequence of dotted arrows is exact, coming from the Mayer--Vietories sequences of the pair $V_1 \cup V_2$, $E_1$, $i$ is the inclusion map, and the map
$$\nu: \left(\widehat{H^q(V_1) \oplus H^q ( V_2)}\right ) \oplus H^q(E_1) \ra \widehat{H^q(e_1) \oplus H^q( f_1)}$$
acts as follows:
$$(l_1, l_2, l_3) \mapsto (l_1|_{e_1}-l_3|_{e_1}, l_2|_{f_1}-l_3|_{f_1}).$$
Therefore we have
\begin{align*}
G^q(W_0) &= \ker \nu \\
&= \{ (l_1, l_2, l_3) \in \left(\widehat{H^q(V_1) \oplus H^q ( V_2)}\right ) \oplus H^q(E_1) \mid l_1|_{e_1}=l_3|_{e_1}, l_2|_{f_1}=l_3|_{f_1}\}\\
&= \{ (l_1, l_2, l_3) \in H^q(V_1) \oplus H^q(V_2) \oplus H^q(E_1) \mid l_1|_{e_1}=l_3|_{e_1}, l_2|_{f_1}=l_3|_{f_1}, l_1|_{\widetilde D} = l_2|_{\widetilde D}\}\\
&= \{ (l_1, l_2, l_3) \in H^q(V_1) \oplus H^q(V_2) \oplus H^q(E_1) \mid l_1, l_2, l_3 \text{ are compatible.}\},
\end{align*}
where `compatible' means that the restrictions of the classes $l_1, l_2, l_3$ to the intersections of $V_1, V_2, E_1$ coincide.
In the general case that $\Sing(Y)$ consists of multiple points ($N \geq 1$), we have
\begin{align*}
G^q(W_0) = \{ (l_1, \cdots, l_{N+2}) \in H^q(V_1) \oplus H^q(V_2) \oplus H^q(E_1) \oplus \cdots \oplus H^q(E_N) \mid \\
l_1, \cdots, l_{N+2} \text{ are compatible.}\}
\end{align*}
With this, one can easily verify that
$$H:=\left( f^*(h)- \frac{1}{2}\sum_i e_i, g^* (f^*(h)- \frac{1}{2}\sum_i f_i ), e_1, \cdots, e_N \right )$$
and
$$H':= \left (m_0, d F, 0, \cdots, 0 \right )$$
belong to $G^2(W_0)$ and $G^{4}(W_0)$ respectively, where
$F \in H^{4}(V_2)$ is a fiber class of the blow up $g:V_2 \ra Y$ over a point in $C$ and $d = m_0 \cdot \widetilde D$.
See the end of \S \ref{sec3} for the definition of $m_0$.
Finally we get the following lemma:
\begin{lem} The lattice $\langle H \rangle$ is isomorphic to $H^2(W_t, \mbz)_f$ with the cup product preserved.
\end{lem}
\begin{proof}
Note that
$$H \cdot H'= \left( f^*(h)- \frac{1}{2}\sum_i e_i \right ) \cdot m_0 + g^* (f^*(h)- \frac{1}{2}\sum_i f_i ) \cdot d F + 0+ \cdots +0 = 1+ 0 = 1.$$
Since $h^{2}(W_t) = 1$ and $h^{2,0}(W_t)=0$, we are done by Theorem \ref{degthm}.
\end{proof}
Note that $\Pic(X) \simeq H^2(X, \mbz) \simeq H^2(W_t, \mbz)$ and $\rk \Pic(X) = 1$. Let $\Pic(X)_f = \langle H_X \rangle$, where $H_X$ be an ample generator. Then we have
\begin{align*}
H_X^3 = H^3 &= \left( f^*(h)- \frac{1}{2}\sum_i e_i \right )^3 + g^* (f^*(h)- \frac{1}{2}\sum_i f_i   )^3+ e_1^3+ \cdots+ e_N^3\\
&= \left(h^3- \frac{1}{8}\sum_i e_i^3 \right ) + \left(h^3- \frac{1}{8}\sum_i f_i^3 \right ) + 1+ \cdots+ 1\\
&= \left(h^3- \frac{1}{8}\sum_i 4 \right ) + \left(h^3- \frac{1}{8}\sum_i 4 \right ) + N\\
&= 2 h^3 -N + N \\
&=2 h^3.
\end{align*}
\begin{theorem} \label{cor4.3}The following map:
$$\pi^* : \Cl(Y) \ra \Pic(X)_f$$
is a bijection.
\end{theorem}
\begin{proof}
For some positive integer $a$, $\pi^*(h) = a H_X$. Note that $a^3 H_X^3 = 2 h^3$ and $H_X^3 = 2h^3$.
So we have $a = 1$ because $H_X^3$ and $h^3$ are positive.  Therefore the map is bijective.
\end{proof}

 Now we attain the goal of this note.

\begin{cor}[Theorem 0.2] \label{lthm}
The following map of lattices
$$\Cl(Y) \ra \Cl(Y_n)_f$$
is injective and primitive.
\end{cor}
\begin{proof}
Note that $X_n:=\pi^{-1}{(Y_n)}$ is an ample divisor of $X$.
Then the induced map $\pi|_{X_n}:X_n \ra Y_n$ is a double covering, branched over $Y_n \cap S$ if  $n$ is even and  $(Y_n \cap S) \cup \Sing(X)$ if $n$ is odd.
Let $i_Y:Y_n \hookrightarrow Y$ and
$i_X: X_n \hookrightarrow X$ be the inclusion maps. By the Lefschetz hyperplane section theorem, the map of lattices
$$i_X^*:\Pic(X)_f \ra \Pic(X_n)_f$$
is injective and primitive.
Consider the following commutative diagram:
$$
\xymatrix{
\Pic(X_n)_f & \Pic(X)_f \ar[l]^{i_X^*}\\
\Cl(Y_n)_f \ar[u]^{\varsigma} &\Cl(Y) \ar[l]^{i_Y^*} \ar[u]_{\pi^*}
}$$
where $\varsigma=(\pi|_{X_n})^*: \Cl(Y_n)_f \ra \Pic(X_n)_f$.
We showed that the map $\pi^* : \Cl(Y) \ra \Pic(X)_f$ is injective and primitive. So $i_X^* \circ \pi^*$ is also injective and primitive. Therefore the map $i_Y^*:\Cl(Y) \ra \Cl(Y_n)_f$ should be injective and primitive.
\end{proof}
The arguments can be generalized to the case of odd dimensional \Q-Fano $n$-folds  with singularities of type $\frac{1}{2}(1,\cdots, 1)$ and other similar conditions.

Now we can remove the condition of the exceptional case $(-K_Y^3, N) = (4,4)$ in Theorem 3.2 of \cite{Le1}. The critical part of the proof of Theorem 3.2 of \cite{Le1} is to show that $a=1$ as in the above proof.  In \cite{Le1}, we relied on the Riemann--Roch theorem and   divisibility of some intersection numbers, which is why  the numbers $(-K_Y^3, N) $ mattered there.
Now we can get the invariants of the Calabi--Yau double cover of Takagi's \Q-Fano $3$-fold of $(-K_Y^3, N) = (4,4)$. The invariants are
$$(H_X^3, c_2 (X) \cdot H_X, {{e}} (X) ) = (8, 44, -88),$$
where $e(X)$ and $c_2(X)$ are the topological Euler characteristic number and the second Chern class of $X$ respectively.
Those invariants verify that it is a new example of Calabi--Yau $3$-folds with Picard number one, which  play an important
role in the moduli spaces of all Calabi--Yau $3$-folds (\cite{Gr}). See Table 1 in \cite{EnSt} for a list of some known ones.

The author would like to express his sincere thanks to the anonymous referee. The referee read the first version very carefully, pointed out many mistakes and made several valuable suggestions.
This work was supported by National Research Foundation of Korea Grant funded by the Korean
Government (2010-0015242) and  Hongik University Research Fund.


\begin{thebibliography}{[CaOsK]}
\bibitem[CaFl]{CaFl} F. Campana and H. Flenner,
\textit{Projective threefolds containing a smooth rational surface
with ample normal bundle},
J. reine angew. Math {\bf 440} (1993), 77--98.

\bibitem[EnSt] {EnSt} C. van Enckevort and D. van Straten,
\textit{Monodromy calculations of fourth order equations of Calabi--Yau type, Mirror symmetry. V},
539--559, AMS/IP Stud. Adv. Math., {\bf 38}, Amer. Math. Soc., Providence, RI, 2006.


\bibitem[Gr]{Gr} M. Gross,
\textit{Primitive Calabi--Yau threefolds},
 J. Diff. Geom. {\bf 45} (1997), 288--318.



\bibitem[Le1]{Le1} N.-H. Lee,
\textit{Calabi--Yau coverings over some singular varieties and new Calabi--Yau $3$-folds with Picard number one}, Manuscripta Math. {\bf125} (2008), no. 4, 531--547.
\bibitem[Le2]{Le2} N.-H. Lee,
\textit{Calabi--Yau construction by smoothing normal crossing varieties},
Int. J. of Math. {\bf 21} (2010), no. 6, 701--725.

\bibitem[Sa1]{Sa1} T. Sano,
\textit{Classification of non-Gorenstein \Q-Fano $d$-folds of Fano index greater than $d-2$},
 Nagoya Math. J. {\bf 142} (1996), 133-143.

\bibitem[Sa2]{Sa2}T. Sano,
\textit{On classification of non-Gorenstein $Q$-Fano $3$-folds of Fano index $1$},
J. Math. Soc. Japan {\bf 47} (1995), no. 2, 369--380.

\bibitem[Ta1]{Ta1} H. Takagi,
\textit{On classification of $\Bbb Q$-Fano $3$-folds of Gorenstein index 2. I},
Nagoya Math. J.
{\bf 167} (2002), 117-155.
\bibitem[Ta2]{Ta2} H. Takagi,
\textit{On classification of $\Bbb Q$-Fano $3$-folds of Gorenstein index 2. II},
Nagoya Math. J.
{\bf 167} (2002), 157-216.
\end{thebibliography}
\end{document}